\documentclass{amsart}
\usepackage{amsmath}
\usepackage{amssymb}
\usepackage{mathtools}
\usepackage{tikz}
\usepackage{newtxtext,newtxmath}
\usepackage{graphicx}

\usepackage{amsthm}
\usepackage[utf8]{inputenc}
\newtheorem{theorem}{Theorem}
\newtheorem{lemma}[theorem]{Lemma}
\newtheorem{definition}[theorem]{Definition}
\newcommand{\Z}{\mathbb{Z}}	
\newcommand{\Q}{\mathbb{Q}}	
\newcommand{\T}{\mathcal{T}}
\newcommand{\expect}{\mathbb{E}}
\newcommand{\prob}{\mathbb{P}}

\DeclareMathOperator{\Aut}{Aut}

\title{Topology and geometry of random $2$-dimensional hypertrees}
\author{Matthew Kahle}
\address[M.K.]{The Ohio State University, 100 Math Tower, 231 W 18th Ave, Columbus, OH 43210}
\email{mkahle@math.osu.edu}
\thanks{M.K.\ was supported in part by NSF-CCF grants \#1740761 and \#1839358. He is grateful to the Simons Foundation for a Simons Fellowship, and to the Berlin Mathematical School for a Mercator Fellowship.}
\author{Andrew Newman}
\address[A.N.]{Technische Universit\"at Berlin, Chair of Discrete Mathematics / Geometry, Strasse des 17. Juni 136, 10623 Berlin, Germany}
\email{newman@math.tu-berlin.de}
\thanks{A.N. was supported by Deutsche Forschungsgemeinschaft (DFG, German Research 
Foundation) Graduiertenkolleg 2434 "Facets of Complexity".}
\date{\today}

\begin{document}

\maketitle

\begin{abstract}
A hypertree, or $\mathbb{Q}$-acyclic complex, is a higher-dimensional analogue of a tree. We study random $2$-dimensional hypertrees according to the determinantal measure suggested by Lyons. We are especially interested in their topological and geometric properties. We show that with high probability, a random $2$-dimensional hypertree $T$ is apsherical, i.e.\ that it has a contractible universal cover. We also show that with high probability the fundamental group $\pi_1(T)$ is hyperbolic and has cohomological dimension $2$.
\end{abstract}

\section{Introduction}

The following enumerative formula is well known.

\begin{theorem} \label{thm:Cayley}
The number of spanning trees on $n$ vertices is $$n^{n-2}.$$
\end{theorem}
The trees are understood to be labelled, i.e.\ on vertex set $[n]:=\{1, 2, \dots, n \}$, and not merely up to isomorphism type. The example $n=4$ is illustrated in Figure \ref{fig:Cayley}. There are only $2$ trees on $4$ vertices up to isomorphism, but there are $16$ labelled trees.

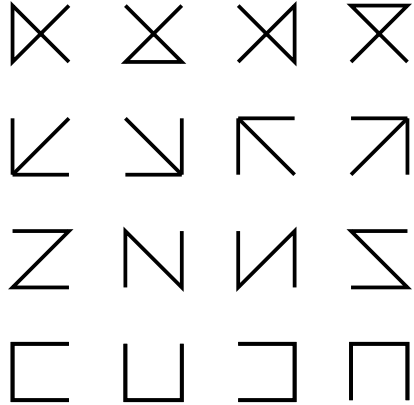
\begin{figure}[t]
\centering
\begin{tikzpicture}[line width=0.5mm,scale=0.75]
\draw (1,0)--++(-1,0)--++(0,1)--++(1,0);
\draw (2,1)--++(0,-1)--++(1,0)--++(0,1);
\draw (4,0)--++(1,0)--++(0,1)--++(-1,0);
\draw (6,0)--++(0,1)--++(1,0)--++(0,-1);

\draw (0,3)--++(1,0)--++(-1,-1)--++(1,0);
\draw (2,2)--++(0,1)--++(1,-1)--++(0,1);
\draw (4,3)--++(0,-1)--++(1,1)--++(0,-1);
\draw (6,2)--++(1,0)--++(-1,1)--++(1,0);

\draw(0,4)--++(1,0);
\draw(0,4)--++(0,1);
\draw(0,4)--++(1,1);

\draw(3,4)--++(-1,0);
\draw(3,4)--++(0,1);
\draw(3,4)--++(-1,1);

\draw (4,5)--++(1,0);
\draw (4,5)--++(1,-1);
\draw (4,5)--++(0,-1);

\draw (7,5)--++(0,-1);
\draw (7,5)--++(-1,-1);
\draw (7,5)--++(-1,0);

\draw (1,0)--++(-1,0)--++(0,1)--++(1,0);
\draw (6,0)--++(0,1)--++(1,0)--++(0,-1);
\draw (2,1)--++(0,-1)--++(1,0)--++(0,1);
\draw (4,0)--++(1,0)--++(0,1)--++(-1,0);

\draw (1,6)--(0,7)--(0,6)--(1,7);
\draw (2,7)--(3,6)--(2,6)--(3,7);
\draw (4,7)--(5,6)--(5,7)--(4,6);
\draw (6,6)--(7,7)--(6,7)--(7,6);

\end{tikzpicture}
\caption{The $4^2=16$ trees on $4$ vertices.}
\label{fig:Cayley}
\end{figure}

Apparently, Theorem \ref{thm:Cayley} was first proved by Borchardt in 1860 \cite{Borchardt1860}. Cayley extended the statement in 1889 \cite{Cayley1889}, and it is often known as ``Cayley's formula.'' Several proofs can be found in Aigner and Ziegler's book \cite{AZ18}. Aigner and Ziegler write that the ``most beautiful proof of all'' was given by Avron and Dershowitz \cite{AD16}, based on ideas of Pitman.

The definition of a tree is that it is connected and has no cycles. Equivalently, a graph $G$ is a tree if it has no nontrivial homology, i.e.\ if $\tilde{H}_0(G) = H_1(G) = 0$. Kalai suggested the topological notion of $\Q$-acyclic simplicial complexes as higher-dimensional analogues of trees in \cite{Kalai1983}. $\Q$-acyclic complexes are sometimes called hypertrees. Here, we use the term \emph{$2$-tree} for a $2$-dimensional hypertree. The precise definition is as follows.

\begin{definition}
We say that a finite $2$-dimensional simplicial complex $S$ is a $2$-tree if it has all of the following properties.
\begin{itemize}
    \item $S$ has complete $1$-skeleton, i.e.\ if the underlying graph is a complete graph. 
    \item $H_1(S;\Q) = H_2 (S; \Q) = 0$.
\end{itemize}
\end{definition}

Kalai proved a general formula for a weighted enumeration of $\Q$-acyclic complexes, which specializes to the following in the case of $2$-trees.

\begin{theorem}[Kalai \cite{Kalai1983}]
$$\sum_{S \in \T(n)} |H_{1}(S)|^2 = n^{\binom{n-2}{2}}$$
\end{theorem}

Here the notation $|G|$ denotes the order of the group $G$. Since $H_{1}(S; \Q) = 0$ by definition, by the universal coefficient theorem we have that $H_{1}(S)$ is a finite group for every $S \in \T(n)$.

The smallest topologically nontrivial example of a $2$-tree is the $6$-vertex projective plane, illustrated in Figure \ref{fig:RP2}. A topological space is said to be \emph{aspherical} if it has a contractible universal cover. The $6$-vertex projective plane is a good example to show that $2$-trees are not always aspherical.

\begin{figure}

\begin{tikzpicture}[scale=1.25]
\draw (90:1)--(210:1)--(330:1)--cycle;
\draw (30:2)--(90:2)--(150:2)--(210:2)--(270:2)--(330:2)--cycle;
\draw (330:1)--(30:2)--(90:1);
\draw (90:1)--(150:2)--(210:1);
\draw (210:1)--(270:2)--(330:1);
\draw (90:1)--(90:2);
\draw (210:1)--(210:2);
\draw (330:1)--(330:2);
\node at (90:0.7) {\bf 1};
\node at (330:0.7) {\bf 2};
\node at (210:0.7) {\bf 3};
\node at (90:2.2) {\bf 4};
\node at (30:2.2) {\bf 5};
\node at (-30:2.2) {\bf 6};
\node at (-90:2.2) {\bf 4};
\node at (-150:2.2) {\bf 5};
\node at (-210:2.2) {\bf 6};
\end{tikzpicture}
\caption{The smallest topologically nontrivial $2$-tree is the $6$-vertex projective plane.}
\label{fig:RP2}
\end{figure}
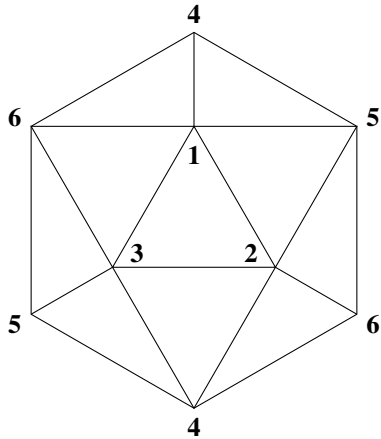

More general enumerative formulas were given by Duval, Klivans, and Martin \cite{DKM09}. These generalizations again are weighted enumeration formulas. In fact, currently for $d \geq 2$ establishing the unweighted enumeration for $d$-trees is an open problem. The best-known upper and lower bounds on unweighted enumeration are given by Linial and Peled \cite{LP19}.

Kalai's enumeration suggests a natural probability distribution on $2$-trees, first studied by Lyons \cite{Lyons09}. Let $\T (n)$ denote the set of all $2$-trees on vertex set $[n]$. Define a probability measure on $\T(n)$ by making the probability of every $2$-tree $T$ proportional to $|H_1(T)|^2$. Equivalently, by Kalai's formula, the probability of any particular $2$-tree $T$ is given by
$$\mathbb{P}(T) = \frac{ |H_1(T)|^2}{n^{\binom{n-2}{2}}}.$$
This is the distribution we study for the rest of this paper. This distribution is in many ways nicer than the uniform distribution. The most important property of this probability distribution for our applications is that it satisfies negative association. This is an a result of Lyons \cite{Lyons09} that we review in Section \ref{sec:neg}. 

We write $T \sim \T(n)$ to denote a $2$-tree chosen according to the determinantal measure described above. For any property $P_n$, we say that property $P_n$ occurs \emph{with high probability (w.h.p.)} if $\mathbb{P}[ T \in P_n] \to 1$ as $n \to \infty$. We are mostly interested in topological and geometric properties of $T$. Our main results are that w.h.p., $T$ is aspherical and that $\pi_1(T)$ is a hyperbolic group of cohomological dimension $2$. The proofs depend on combining ideas from probability, topology, and geometric group theory.

We note that many other models of random simplicial complex have been studied --- see, for example, the survey in Chapter 22 of \cite{Handbook2018}. The closest model to what we study here is the Linial--Meshulam model $Y \sim Y(n,p)$ introduced in \cite{LM06}, which is the ``face-independent'' model, a higher-dimensional analogue of the Erd\H{o}s--R\'enyi edge-independent random graph $G(n,p)$. Negative association allows us to relate random $2$-trees with $Y(n,p)$. Babson, Hoffman, and Kahle showed the fundamental group $\pi_1(Y)$ (in a certain range of parameter) is hyperbolic, with high probability, and Costa and Farber showed that $Y$ is ``almost'' aspherical, and they also showed that $\pi_1(Y)$ has cohomological dimension $2$. 

The remainder of the paper is organized as follows. In Section \ref{sec:neg}, we review definitions of ``determinantal measures'' and negative association. In Section \ref{sec:hyper}, we show that w.h.p.\ the fundamental group of the random $2$-tree $\pi_1(T)$ is a hyperbolic group w.h.p.  In Section \ref{sec:nontrivial} we show that w.h.p.\ $H_1(T) \neq 0$ and in Section \ref{sec:coho} we show that w.h.p. $\pi_1(T)$ is cohomologically $2$-dimensional. In Section \ref{sec:questions}, we suggest a few questions for future study.

\section{Negative association} \label{sec:neg}

We first review the definitions of determinantal measure and negative association. In particular, we briefly overview the work of Lyons \cite{Lyons03,  Lyons09} which is essential for our results. In \cite{Lyons03}, Lyons defines a \emph{determinantal probability measure} as follows.
\begin{definition}
Given a finite set $E$, a probability measure $\mu$ on $E$ is said to be a determinantal probability measure if there exists a matrix $M$ so that for all $S \subseteq E$, the probability that a subset $T$ sampled by $\mu$ contains $S$ as a subset is given by $\det(M_{S, S})$, i.e.\ the determinant of the submatrix of $M$ whose rows and columns are indexed by $S$.
\end{definition}

A \emph{monotone increasing event} is an event $\mathcal{A}$ so that $S \in \mathcal{A}$ and $S \subseteq T$ implies that $T \in \mathcal{A}$. A key fact about determinantal measure is that they satisfy \emph{negative association}, defined in \cite{Lyons09} as follows.

\begin{definition}
Given a finite set $E$, a probability measure $\mu$ on $E$ is said to satisfy negative association provided that for every pair of monotone increasing events $\mathcal{A}$ and $\mathcal{B}$
\[\mu(\mathcal{A} \cap \mathcal{B}) \leq \mu(\mathcal{A})\mu(\mathcal{B}).\]
\end{definition}

Lyons shows in Theorem 6.5 of \cite{Lyons03} that determinantal probability measures satisfy negative association, and in \cite{Lyons09} that the torsion-squared distribution on $2$-trees we consider here is a determinantal measure. 

For our purposes here, we will primarily be interested in the situation where we wish to bound the probability that a determinantal-measure sampled $2$-tree contains some particular, finite subcomplex. In our case, given $n$, the set $E$ in the definition of a determinantal probability measure is the set of all $\binom{n}{3}$ triangles on $n$ vertices. By Euler characteristic any $2$-tree contains exactly $\binom{n - 1}{2}$ triangles, so we have by symmetry that under the torsion-squared distribution the probability that a random $2$-tree contains any particular face is \[\frac{\binom{n-1}{2}}  {\binom{n}{3}} = \frac{3}{n}.\]  Thus, for a fixed (labeled) subcomplex $K$ given by triangles $\sigma_1, ..., \sigma_k$, we have by negative association that the probability that $T$ sampled from the torsion-squared distribution contains $K$ as a subcomplex is at most $(3/n)^k$. For this reason we say that the faces of a torsion-squared random $2$-tree are negatively correlated. 

In contrast to the determinantal measure, the uniform measure on $2$-trees need not have negatively correlated faces. This can be seen by a exhaustive enumeration of $2$-trees on 6 vertices. This is discussed in \cite{KLNP18}, and we review the discussion as follows. There are 46620 $2$-trees on vertex set $\{1, ..., 6\}$. As $2$-trees on 6 vertices contain $\binom{5}{2} = 10$ triangles out of a total of $\binom{6}{3} = 20$ possible triangles, by symmetry we have that the probability that a uniform random $2$-tree contains any given triangle is $10/20=1/2$. On the other hand, $11664$ $2$-trees contain both the triangle $[1, 2, 3]$ and the triangle $[4, 5, 6]$ by exhaustive enumeration. However $11664/46620 > 1/4.$ Changing to the torsion-squared distribution resolves this in the case $n = 6$ because 12 of the $2$-trees on 6 vertices are labeled triangulations of the projective plane. None of these contain both $[1, 2, 3]$ and $[4, 5, 6]$. Sampling by torsion-squared counts these 12 complexes each 4 times and gives that the probability a $2$-tree contains both $[1, 2, 3]$ and $[4, 5, 6]$ is $11664/46656 = 1/4$.

\section{Hyperbolicity} \label{sec:hyper}

We show in this section that  w.h.p.\ $\pi_1(T)$ is hyperbolic in the sense of Gromov \cite{Gromov87}. The proof is based on the main result in \cite{BHK11}---indeed, we will use a key lemma from the paper as our main tool. We first review a few key definitions and notions related to hyperbolicity.

Let $C_r$ denote a cycle of length $r$. For a simplicial complex $X$, a \emph{loop} is a simplicial map $\gamma: C_r \to X$. In this case, we define the \emph{length of $\gamma$} by $L(\gamma)=r$.

We say that $(C_r \xrightarrow{b} D \xrightarrow{\pi} X)$ is a \emph{filling} of $\gamma$ if $D$ is a simplicial complex, $b$ and $\pi$ are simplicial maps such that $\gamma = \pi b$, and the mapping cylinder of $b$ is homeomorphic to a $2$-dimensional disk.

\begin{figure}
    \centering
    \begin{tikzpicture}[scale=1.25]
    \draw [line width=0.75mm] (0,0)--(3,0);
    \draw [line width=0.75mm,fill=lightgray](3,0)--(3.866,0.5)--(3.866,-0.5)--cycle;
    \draw [line width=0.75mm,fill=lightgray](0,0)--(-0.866,0.5)--(-0.866,-0.5)--cycle;
    \draw (-1.732,2)--(0,1)--(3,1)--(4.732,2)--(4.732,-2)--(3,-1)--(0,-1)--(-1.732,-2)--cycle;
    \draw (-1.732,2)--(-0.866,0.5);
    \draw (0,1)--(0,0);
    \draw (1,1)--(1,0);
    \draw (2,1)--(2,0);
    \draw (3,1)--(3,0);
    \draw (4.732,2)--(3.866,0.5);
    \draw (4.732,-2)--(3.866,-0.5);
    \draw (3,-1)--(3,0);
    \draw (2,-1)--(2,0);
    \draw (1,-1)--(1,0);
    \draw (0,-1)--(0,0);
    \draw (-1.732,-2)--(-0.866,-0.5);
    \end{tikzpicture}
    \caption{The mapping cylinder of $b$ for a filling $(C_r \xrightarrow{b} D \xrightarrow{\pi} X)$ of a cycle $\gamma$. Here $L(\gamma)=12$ and $A(\gamma)=2$.} 
    \label{fig:my_label}
\end{figure}
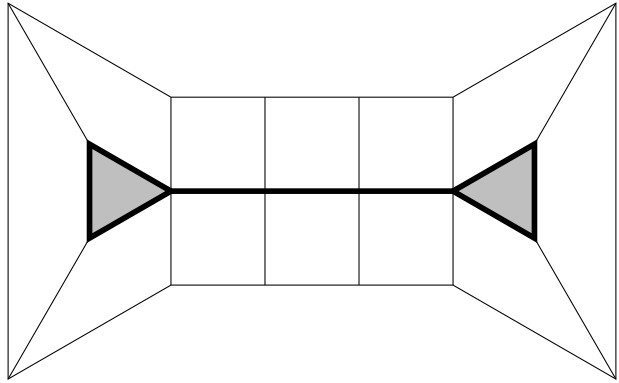
Let $f_2(D)$ denote the number of $2$-dimensional faces in $D$. We define the area of the filling to be the number of faces in $f_2(D)$. For a null-homotopic loop $\gamma$, we say that the \emph{area of $\gamma$}, denoted $A(\gamma)$, is the minimal area over all fillings.

Now, we are ready for a definition of hyperbolic group.

\begin{definition}
Let $\Delta$ be a finite simplicial complex. We say that the fundamental group $\pi_1 ( \Delta)$ is hyperbolic if there exists a constant $K > 0$ such that 
$$A( \gamma) \le K L (\gamma)$$
for every null-homotopic loop $\gamma$.
\end{definition}

It is not obvious from this definition, but this is an invariant property of the group $\pi_1( \Delta)$ which does not depend on the choice of simplicial complex $\Delta$. This definition in terms of a linear isoperimetric inequality is similar to the first definition given by Gromov in 
\cite{Gromov87}. Satisfying such an inequality is equivalent to a Cayley graph of the group being $\delta$-hyperbolic, or the group being word hyperbolic.

Our main tool in this section is the following, which appears in \cite{BHK11}. 

\begin{theorem}[Babson--Hoffman--Kahle,  Theorem 1.9 in \cite{BHK11}]\label{thm:linear}  Let $\epsilon > 0$, and suppose that $\Delta$ is a finite simplicial complex such that for every subcomplex $S \subseteq \Delta$, we have that
$$\frac{f_2(S)}{f_0(S)} \le 2 - \epsilon.$$
Then $\Delta$ satisfies a linear isoperimetric inequality. Namely
$$A(\gamma) \le \lambda \, L(\gamma)$$
for every null-homotopic loop $\gamma$. Here $\lambda = \lambda( \epsilon)$ is a constant which only depends on $\epsilon$.
\end{theorem}

We also require the following, which allows us to pass from local to global isoperimetric inequalities. This particular statement for simplicial complexes and its proof also appear in \cite{BHK11}, and it is based on earlier work of Gromov \cite{Gromov87} and Papasoglu \cite{Papasoglu96}.

\begin{theorem} \label{thm:local-to-global}
Suppose that $\rho \ge 1$ and $X$ is a finite simplicial complex for which every null-homotopic loop $\gamma: C_r \to X$ with $A(\gamma) \le 44^3 \rho^2$ satisfies $A(\gamma) \le \rho  L(\gamma)$. Then every null-homotopic loop $\gamma: C_r \to X$ satisfies $A(\gamma) \le 44 \rho L(\gamma)$.
\end{theorem}

In other words, if $X$ satisfies a linear isoperimetric inequality locally, then it satisfies one globally, although perhaps with a worse isoperimetric constant. So it suffices to check hyperbolicity on balls of finite radius. We are now ready to prove the main result of the section.

\begin{theorem} \label{thm:hyperbolic} Suppose $T \sim \T (n)$ is a random $2$-tree according to the determinantal measure. Then w.h.p. $\pi_1 (T)$ is a hyperbolic group. 
\end{theorem}

\begin{proof}[Proof of Theorem \ref{thm:hyperbolic}]

With foresight into the calculations to come, let $\epsilon = 1/2$, and let $\lambda = \lambda( \epsilon)$ be the constant guaranteed by Theorem \ref{thm:linear}. So for every finite simplicial complex $\Delta$ satisfying the condition of Theorem \ref{thm:linear}, and every null-homotopic loop $\gamma: C_r \to \Delta$, we have 
$$A( \gamma) \le \lambda L(\lambda).$$

Now, let $C$ be chosen such that $$C \ge \max\{ 44^3 \lambda^2, 44^3 \},$$
and then let $C'$ be chosen such that
$$C' \ge  \frac{C}{2} \left( 1 + \frac{1}{\lambda} \right) + 1.$$
We emphasize that $C$ and $C'$ are chosen to be sufficiently large, but are still fixed as $n \to \infty$.

First, we check that w.h.p.\ for every subcomplex $\Delta \subset T$ on at most $C'$ vertices, we have 
$$\frac{f_2(\Delta)}{f_0(\Delta)} < \frac{3}{2}.$$
Note first that if there exists a subcomplex $\Delta \subset T$ with $f_2(\Delta) \ge (3/2) f_0(\Delta),$
then there exists a subcomplex $\Delta'\subset T$ with $f_2(\Delta') = \lceil (3/2) f_0(\Delta')\rceil$. Indeed, $\Delta'$ can be obtained by deleting one face from $\Delta$ at a time until equality is achieved.

A union bound, together with negative association, gives that
$$ \prob \left[ \exists \, \Delta \subset T \mbox{ with } f_2(\Delta) > (3/2) f_0(\Delta) \right] \le \sum_{k=1}^{C'} \binom{n}{k} \binom{\binom{k}{3}}{\lceil (3/2)k \rceil} \left(\frac{3}{n}\right)^{\lceil (3/2)k \rceil}.
$$
The sum tends to zero as $n$ tends to infinity, since $C'$ is fixed so there are only a bounded number of summands, and every summand tends to zero.

By Theorem \ref{thm:linear}, we have that w.h.p.\ every subcomplex $\Delta \subset T$ on at most $C'$ vertices satisfies the linear isoperimetric inequality
$$A(\gamma) \le \lambda L(\gamma).$$

Next, we check that this implies that
$$A(\gamma) \le \lambda \, L(\gamma)$$ 
for every null-homotopic loop $\gamma$ in $T$ with $A(\gamma) \le C$.

Suppose that $\gamma$ is a null-homotopic loop with $A(\gamma) \le C$. 
If $L( \gamma ) >  C / \lambda$, then since $A( \gamma ) \le C$ it is immediate that $A( \gamma ) \le \lambda L( \gamma )$.

So suppose instead that $L( \gamma ) \le  C / \lambda$. In this case, $A(\gamma)$ and $L(\gamma)$ are both bounded. It follows that if $(C_r \xrightarrow{b} D \xrightarrow{\pi} T)$ is a filling of $\gamma$, then the number of vertices $f_0(D)$ is bounded as well. Indeed, let $v$, $e$, and $f$ denote the number of vertices, edges, and faces in the mapping cylinder of $b$. Since we have a bijection between vertices of the mapping cylinder, and the disjoint union of vertices in $C_r$ and vertices in $D$, we have $$v = L(\gamma) + f_0(D).$$ 
By double counting edge-face incident pairs have $2e = 5L + 3A$. or $$e = (5/2) L + (3/2) A.$$
Finally, we have
$$f = L + A,$$
since every face of the mapping cylinder is either a square face (corresponding to a single edge of $C_r$) or a triangle face of the simplicial complex $D$.
Since the mapping cylinder is a topological disk, we have $$v-e+f=1.$$ Putting it all together gives that $ f_0(D)  = A(\gamma)/2 +  L(\gamma)/2 +1$.

In the case we are interested in, we have
\begin{align*}
    f_0(D) & = A(\gamma)/2 +  L(\gamma)/2 +1 \\
    & \le \frac{C}{2} + \frac{C}{2 \lambda} + 1\\
    & = \frac{C}{2} \left( 1 + \frac{1}{\lambda} \right) + 1\\
    & \le C',
\end{align*}
by choice of $C'$. Then the image of the map $\pi:  D \to T$ lies in subcomplex $\Delta \subset T$ on at most $C'$ vertices, so by the above $A(\gamma) \le \lambda L(\gamma)$, as desired. 

Let $\rho = \max \{ 1, \lambda \}$. Then $\rho \ge 1$ and we have that  $A(\gamma) \le \rho L(\gamma)$  for every null-homotopic loop $\gamma$ with $A(\gamma) \le C$. Theorem \ref{thm:local-to-global} gives that 
$$A( \gamma) \le 44 \rho L( \gamma)$$
for all null-homotopic $\gamma$ in $T$. Setting $K = 44 \rho$, we have the desired result.

\end{proof}

\section{Nontriviality and expected order of torsion} \label{sec:nontrivial}

In this section, we give upper bounds on the probability that homology $H_1(T)$ is trivial and lower bounds on its expected order. We make use of the following observation of Kalai \cite{Kalai1983} on the number of $2$-trees on $n$ vertices. 

\begin{lemma} \label{lem:numb}
Let $N(n)$ denote the number of $2$-trees on $n$ vertices. Then 
$$N(n) \le (en/3)^{\binom{n-1}{2}}.$$
\end{lemma}

\begin{proof}
Every $2$-tree $T$ on $n$ vertices has $\binom{n}{2}$ edges. The Betti numbers are $\beta_0 = 1$ and $\beta_1 = \beta_2 = 0$, by definition. 
By the Euler formula, $T$ has $\binom{n-1}{2}$ $2$-dimensional faces.
So the total number of $2$-trees is at most 
$$    \binom{ \binom{n}{3}}{\binom{n - 1}{2}}  \le \left( \frac{e \binom{n}{3}}{\binom{n - 1}{2}} \right)^{\binom{n-1}{2}} 
     = \left( \frac{en}{3} \right)^{\binom{n-1}{2}}.
$$
\end{proof}

\begin{theorem}\label{nontrivialH1}
Let $T \in \T(n)$. With probability at least $1 - \exp(-\Omega(n^2))$, we have $H_1(T) \neq 0$.
\end{theorem}
\begin{proof}
The probability that $X$ sampled from $\mathcal{C}_n$ with respect to the determinantal measure has $H_1(X) = 0$ is 
\begin{eqnarray*}
\sum_{\{X | H_1(X) = 0\}} \frac{1}{n^{\binom{n - 2}{2}}} &=& \frac{\text{Number of $2$-trees with $H_1(X) = 0$}}{n^{\binom{n - 2}{2}}}\\
&\leq& \frac{N(n)}{n^{\binom{n - 2}{2}}}
\end{eqnarray*}
By Lemma \ref{lem:numb}, we have that the above is at most
\begin{eqnarray*}
\left(\frac{en}{3}\right)^{\binom{n - 1}{2}} \frac{1}{n^{\binom{n - 2}{2}}} =  \left(\frac{e}{3}\right)^{\binom{n - 1}{2}} n^{n - 2}.
\end{eqnarray*}

It follows that the probability that $H_1(T)=0$ is $\exp(-\Omega(n^2))$.
\end{proof}

Next, we prove the following. 

\begin{theorem}\label{thm:expect} We have that $$\expect \left[ |H_1(T) | \right] \ge \left( \frac{3}{e} \right)^{\binom{n-2}{2}} \left( \frac{3}{en} \right)^{n-2}.$$
So in particular, we have that
$$\expect [ |H_1(T)| ] = \exp \left( \Theta \left( n^2 \right) \right).$$
\end{theorem}

We will use the following inequality.

\begin{lemma} \label{lem:ineq}
Let $x_1, x_2, \dots x_k \ge 0$ be non-negative real numbers.

Then it follows that
$$ \sum_{i=1}^k x_i^3 \ge  \frac{1}{\sqrt{k}} \left( \sum_{i=1}^k x_i^2 \right)^{3/2}.$$
\end{lemma}

\begin{proof}[Proof of Lemma \ref{lem:ineq}]
Jensen's inequality tells us that for a convex function $\phi$, numbers in its domain $y_1, y_2, \dots, y_k$, and positive weights $a_1, a_2, \dots, a_k$, we have 
$$\phi \left( \frac{\sum_{i=1}^k a_i y_i }{\sum_{i=1}^k a_i} \right) \le \frac{\sum_{i=1}^k a_i \phi(y_i) }{\sum_{i=1}^k a_i}.$$

Set $a_i = 1$ and $y_i = x_i^2$ for $i=1, 2, \dots k $, and let $\phi(x) = x^{3/2}$. We note that $\phi(x)$ is convex on the domain $\{ x \mid x \ge 0\}$.
\end{proof}

Given the lemma, we  prove Theorem \ref{thm:expect}.

\begin{proof}[Proof of Theorem \ref{thm:expect}]
By definition, we have that
\begin{align*}
    \expect \left[ |H_1(T)| \right] 
    &= \sum_{T \in \T(n)} \prob[T] |H_1(T)| \\
    & =  \sum_{T \in \T(n)} \frac{|H_1(T)|^2}{ n^{ \binom{n-2}{2}} }|H_1(T)| \\
    & = \frac{1}{ n^{ \binom{n-2}{2}} } \sum_{T \in \T(n)} |H_1(T)|^3 \\
    & \ge \frac{n^{3/2} \binom{n-2}{2}}{n^{\binom{n-2}{2}} \left( en/3 \right)^{(1/2)\binom{n-1}{2} }}. \\
\end{align*}
This last step is by applying Lemmas \ref{lem:numb} and \ref{lem:ineq}.
Simplifying, we have that
\begin{align*} \expect \left[ |H_1(T)| \right] & \ge \left( \sqrt{\frac{3}{e}} \right) ^{\binom{n-2}{2}} \left(\sqrt{ \frac{3}{en}} \right)^{n-2}\\
& = \left( (3/e)^{1/4}-o(1) \right)^{n^2}.\\
\end{align*}

\end{proof}

This is on the scale of the largest torsion possible, in the sense that for \emph{every} simplicial complex $\Delta$ on $n$ vertices, we have that the order of the torsion part of homology is bounded by 
$$    |H_1(\Delta)_{\mbox{torsion}}| \le \sqrt{3}^{\binom{n-1}{2}}\le \left(  3^{1/4} - o(1)  \right)^{n^2}, $$
This upper bound on torsion appears in many places, including \cite{Soule99} and \cite{HKP17}, and perhaps first appeared in Kalai's weighted enumeration of hypertrees \cite{Kalai1983}.

\section{$T$ is aspherical and $\pi_1(T)$ has cohomological dimension $2$} \label{sec:coho}

The main result of this section is the following.

\begin{theorem}\label{thm:aspherical}
Let $T \sim \T(n)$. Then, w.h.p.\ $T$ is aspherical.
\end{theorem}

Our proof will use the following theorem of Costa and Farber \cite{CF15-1}. It is worth noting that this is a purely topological and combinatorial statement, and does not involve probability. 

\begin{theorem}[Costa--Farber, Theorem 11 of \cite{CF15-1}]\label{CFDeterministic}
There exists a finite list $\mathcal{L}$ of compact $2$-dimensional complexes with the following two properties:
\begin{enumerate}
    \item A finite simplicial 2-complex $Y$ is aspherical if it contains no subcomplex isomorphic to a complex from the list $\mathcal{L}$.
    \item For any $S \in \mathcal{L}$ other than the boundary of the tetrahedron, there exists a subcomplex $S' \subseteq S$ with $f_0(S')/f_2(S') \leq 46/47$ 
    \end{enumerate}
\end{theorem}
We've modified the statement slightly from its original form. In \cite{CF15-1}, Costa and Farber show that for a certain regime of $p$, $Y(n, p)$ is \emph{asphericable}, that is it has the property that after removing a single face from every embedded tetrahedron boundary the resulting complex is aspherical. In the original formulation, the conclusion of part (1) is that the complex is asphericable. Here we simply added the tetrahedron boundary to the set $\mathcal{L}$ as we already know that a $2$-tree $T$ cannot contain tetrahedron boundaries, since $H_2(T; \Q) = 0$.


\begin{proof}[Proof of Theorem \ref{thm:aspherical}]
Take $\mathcal{L}$ to be the finite list of complexes in Theorem \ref{CFDeterministic}. We show that with high probability $T \sim \T(n)$ contains no subcomplex in $\mathcal{L}$. We already know that $T$ cannot contain the boundary of a tetrahedron; for any other $S \in \mathcal{L}$, we bound the probability that a determinantal-measure random $2$-tree contains $S$. For $S \in \mathcal{L}$, different from the tetrahedron boundary, take $S'$ to be a subcomplex of $S$ satisfying condition (2) of Theorem \ref{CFDeterministic} and let $v$ denote $f_0(S')$, then the probability that $T \sim \T(n)$ contains $S$ is at most the probability that it contains $S'$. By negative correlation the probability that $T$ contains $S'$ is at most
\[ \binom{n}{v} v! \left(\frac{3}{n}\right)^{47v/46}. \]
Indeed to embed $S'$ in $T$ we have to choose the $v$ vertices and then we have $|\Aut(S')| \leq v!$ ways to choose a copy of $S'$ on the selected vertex set. Now by negative correlation the probability that every face of the selected copy of $S'$ appears in $T$ is at most the product of the probability that each face of $S'$ appears, thus it is at most $(3/n)^{f(S')} \leq (3/n)^{47v/46}$. As $v$ is fixed and at least one the probability that $T$ contains $S'$ as an embedded subcomplex is $O(n^{-1/46})$. By a union bound over the finite list $\mathcal{L}$, the probability that $T$ contains any member of $\mathcal{L}$ is $O(n^{-1/46}) = o(1)$. Thus by part (1) of Theorem \ref{CFDeterministic}, with probability at least $1 - O(n^{-1/46})$, $T \sim \T(n)$ is aspherical.
\end{proof}

For a group $G$, let $\mbox{cd}_R(G)$ denote the cohomological dimension of $G$ with respect to coefficient ring $R$. We have the following immediate consequence of Theorem \ref{thm:aspherical}

\begin{theorem}
Let $T \sim \T(n)$. Then w.h.p.\ $\mbox{cd}_{\Z} \left( \pi_1(T) \right) = 2.$
\end{theorem}

\begin{proof}
In Section \ref{sec:nontrivial}, we saw that w.h.p.\ $H_1(T) \neq 0$ is nontrivial. By definition, $H_1(T, \Q) = 0$, so we have that w.h.p.\ $H_1(T)$ is a nontrivial group, and not a free group. By the Stallings--Swan Theorem \cite{Stallings68, Swan69}, we have $\mbox{cd}_{\Z} \left( \pi_1(T) \right) \ge 2$.

On the other hand, if $T$ is aspherical then $T$ is itself a $2$-dimensional $BG$ for $G = \pi_1(T)$, so $\mbox{cd}_{\Z} \left( \pi_1(T) \right) \le 2$.
\end{proof}

We end with a comment. If  $\mbox{cd}_{\Z} \left( \pi_1(T) \right) = 2$, then w.h.p.\ $\pi_1(T)$ must be infinite. Indeed, w.h.p.\ $\pi_1(T)$ is a nontrivial group, and it can not have any elements of finite order, since this would imply that the cohomological dimension is infinite. Even though our results show that (according to the determinantal measure) almost all hypertrees $T$ have infinite fundamental group $\pi_1(T)$, at the moment we are not aware of any explicit examples. 

\section{Questions} \label{sec:questions}

It seems to us that the random $2$-tree is a natural model for stochastic topology. We suggest a few more questions for further study. \\

\begin{itemize}
    \item {\bf Does $\pi_1(T)$ have Kazhdan's Property (T)?} A group is said to have Property (T) if the trivial representation is an isolated point in the unitary dual equipped with the Fell topology. This is an important property in representation theory, geometric group theory, ergodic theory, and the theory of expander graphs. See the monograph \cite{BdlHV08} for a comprehensive introduction. We conjecture that for $T \sim \T(n)$, w.h.p.\ $\pi_1(T)$ has Property (T). One motivation for the conjecture is that in \cite{HKP19}, it is shown that in the stochastic process version of the Linial--Meshulam  random $2$-complex, as soon as the complex $Y$ is pure $2$-dimensional, $\pi_1(Y)$ has Property (T). In general, it would be interesting to know about ``high-dimensional'' expander properties of random $2$-trees. See Lubotzky's 2018 ICM talk for an overview of high-dimensional expanders \cite{Lubotzky18}. \\
    \item {\bf Is $H_1(T)$ Cohen--Lenstra distributed?} Cohen--Lenstra heuristics, first arising in number-theoretic settings \cite{CL}, are a natural model for random finite abeliean groups. These heuristics now appear in several contexts, including cokernels of random matrices and random graph Laplacians. See, for example, \cite{CKLPW, EVW, Koplewitz, Lengler, NguyenWood, Wood2017}. In \cite{KLNP18}, Kahle, Lutz, Newman, and Parsons studied the \emph{uniform measure} on random $2$-trees, and examined the random finite abelian groups that appeared as the first homology group. There is strong experimental evidence for the conjecture that for any fixed prime $p$, the probability that the Sylow $p$-subgroup of homology $G$ is distributed according to a probability distribution assigning probability inversely proportional to $|\Aut(G)|$. Equivalently, for a given prime $p$ and $p$-group $H$, the probability that $G$ is isomorphic to $H$ is
    given by the formula
    $$\frac{ \prod_{k=1}^{\infty} \left( 1- p^{-k} \right)}{|\mbox{Aut}(H)|}$$
    
    We expect this same limiting probability holds, even if the $2$-trees are sampled by the determinantal measure instead. One can sample a $2$-tree with with the Metropolis--Hastings algorithm, and preliminary experiments support the conjecture. \\
    
    \item {\bf Is there a scaling limit?} The random $2$-tree is a $2$-dimensional analogue of the uniform spanning tree (UST) on the complete graph on $n$ vertices. The UST is known to have a \emph{scaling limit}, where a suitably rescaled UST converges to a limiting distribution as $n \to \infty$. This limit was described by Aldous in \cite{Aldous91a, Aldous91b, Aldous93}, who called it the ``continuum random tree'', and it has been studied extensively since then. An illustration of a continuum random tree computed by by Igor Kortchemski appears in Figure \ref{fig:CRT}. Is there a scaling limit for the random $2$-tree?
    
    \begin{figure}
        \centering
        \includegraphics[width=4.5in]{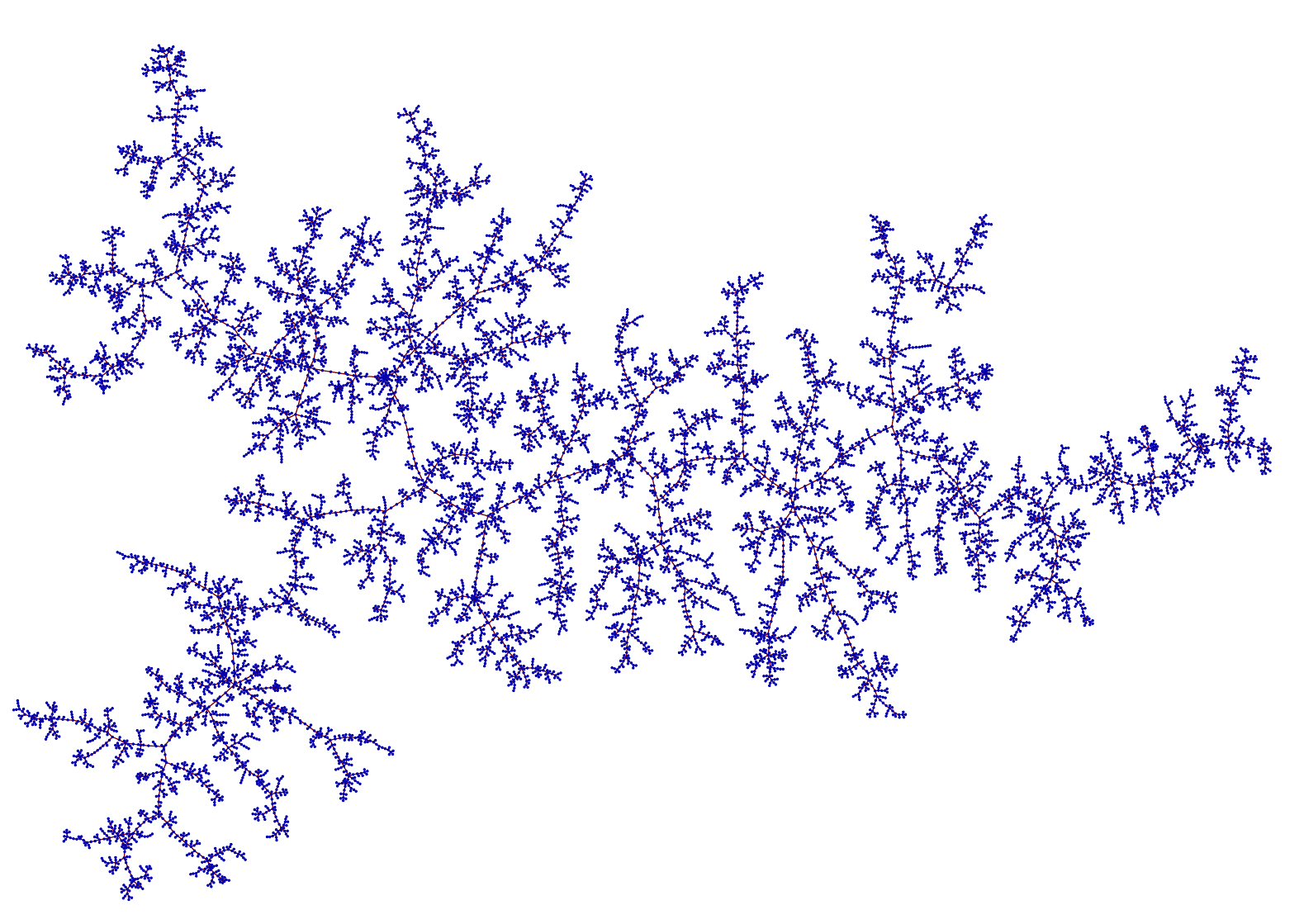}
        \caption{An image of a continuum random tree generated by Igor Kortchemski.}
        \label{fig:CRT}
    \end{figure}
    
\end{itemize}

\section*{Acknowledgements.} M.K.\ is grateful to Nati Linial for suggesting the study of random hypertrees and for encouragement. We thank TU Berlin for hosting us during the 2019--20 academic year. We also gratefully acknowledge Igor Kortchemski's permission for use of the image in Figure \ref{fig:CRT}.

\bibliographystyle{plain}
\bibliography{treerefs}
\nocite{*}

\end{document}